\documentclass[twoside,11pt]{amsart} \usepackage{epsfig}
\usepackage{latexsym} 
\usepackage{amsmath} \usepackage{amssymb}
\usepackage{cancel}



\newtheorem{theorem}{Theorem}[section]

\newdimen\epsfxsize
\newdimen\epsfysize

\newcommand {\gap}     {\makebox[0.075 in]{}}

\newcommand {\fto}     {\longrightarrow}

\newcommand{\pml}[1] {{#1}^\#}

\newcommand{\Z}     {{\mathbb Z}}

\newcommand{\pgap}   {{\mathcal G}}

\begin{document}

\title{Constellations of gaps in Eratosthenes sieve}

\date{25 Feb 2015}

\author{Fred B. Holt}
\address{fbholt62@gmail.com ;  5520 - 31st Ave NE, Seattle, WA 98105}

\begin{abstract}
A few years ago we identified a recursion that works directly with the gaps among
the generators in each stage of Eratosthenes sieve.  This recursion provides explicit
enumerations of sequences of gaps among the generators, which sequences are known as 
constellations.  

Over the last year we identified a discrete linear system that exactly models
the population of any gap across all stages of the sieve.
In August 2014 we summarized our results from analyzing this discrete model on populations
of single gaps.  This paper extends the discrete system to model the populations of constellations
of gaps. 

The most remarkable result is a {\em strong Polignac result on arithmetic progressions}.
We had previously established that the equivalent of Polignac's conjecture holds for Eratosthenes 
sieve -- that every even number arises as a gap in the sieve, and its population
converges toward the ratio implied by Hardy and Littlewood's Conjecture B.
Extending that work to constellations, we 
here establish that for any even gap $g$, if $p$ is the maximum prime
such that $\pml{p} | g$ and $P$ is the next prime larger than $p$, 
then for every $2 \le j_1 < P-1$, the constellation $g,g,\ldots,g$
of length $j_1$ arises in Eratosthenes sieve.  This constellation corresponds to an
arithmetic progression of $j_1+1$ consecutive candidate primes.

\end{abstract}

\maketitle

\section{Introduction}
This paper is a technical note that builds on our previous paper, 
{\em Eratosthenes sieve and the gaps between primes} \cite{HRErat}.
As such, we will assume familiarity with the notations, models,
methods, and results of that work.

A {\em constellation among primes} \cite{Riesel} is a sequence of consecutive gaps
between prime numbers.  Let $s=a_1 a_2 \cdots a_k$ be a sequence of $k$
numbers.  Then $s$ is a constellation among primes if there exists a sequence of
$k+1$ consecutive prime numbers $p_i p_{i+1} \cdots p_{i+k}$ such
that for each $j=1,\ldots,k$, we have the gap $p_{i+j}-p_{i+j-1}=a_j$.  

We do not study the gaps between primes directly.  Instead, we study the cycle of gaps
$\pgap(\pml{p})$ at each stage of Eratosthenes sieve.  Here, $\pml{p}$ is the
{\em primorial} of $p$, which is the product of all primes from $2$ up to and including $p$.
$\pgap(\pml{p})$ is the cycle of gaps among the generators of $\Z \bmod \pml{p}$.
These generators and their images through the counting numbers are the candidate primes
after Eratosthenes sieve has run through the stages from $2$ to $p$.  All of the remaining primes
are among these candidates, but many of the candidates do not survive further stages of the sieve 
to be confirmed as primes.

At times we will speak of constellations {\em corresponding} to gaps between primes.
We do this for intuitive convenience, to motivate why the particular constellations we
study may be of interest.  The results here are for populations within Eratosthenes sieve,
and the constellations correspond to primes only to the extent that they survive the
sieve.

In \cite{HRErat} we identified a discrete dynamic system that exactly models the populations
of gaps from the cycle of gaps $\pgap(\pml{p_k})$ in one stage of Eratosthenes sieve to 
the next, $\pgap(\pml{p_{k+1}})$.  We use the eigenstructure of this system to obtain
asymptotic ratios between the populations of various gaps in the sieve.

With the following theorem, we generalize the discrete dynamic system for gaps to 
model the populations of constellations across stages of Eratosthenes sieve.

\begin{theorem}\label{sThm}
Let $s$ be a constellation of length $j_1$ and of sum $|s|$.
Let $p_0$ be a prime such that $|s| < 2p_1$.

Let $n_j(s,\pml{p})$ be the population of driving terms of length $j$ for $s$ in $\pgap(\pml{p})$,
with $j_1 \le j \le J$ and $p \ge p_0$.  Then
\begin{eqnarray}\label{j1JEq}
n_s(\pml{p_k}) & = & \left. M_{j_1:J}\right|_{p_k} \cdot n_s(\pml{p_{k-1}}) \\
 & = & R \cdot \left. \Lambda_{j_1:J} \right|_{p_k} \cdot L \cdot n_s(\pml{p_{k-1}}). \notag
\end{eqnarray}
in which $L$ is the upper triangular Pascal matrix, $R$ is the upper triangular Pascal
matrix with alternating signs $(-1)^{i+j}$, and $\Lambda$ is the diagonal matrix of
eigenvalues $(p_k-j-1)$ for $j_1 \le j \le J$.
\end{theorem}

As with our work on gaps, this explicit enumeration has two limitations.  The first limitation
is that all of the populations grow super exponentially by factors of $(p-j_1-1)$.  To address this,
we normalize the dynamic system by dividing by factors of $(p-j_1-1)$, which makes the
dominant eigenvalue equal to $1$.

The second limitation of Theorem~\ref{sThm} is the condition $|s| < 2p_1$.  Because of
this condition, to exactly model the population of a constellation we have to calculate its
initial conditions in $\pgap(\pml{p})$, for $p$ very close to $|s|$.  But $\pgap(\pml{p})$ 
 consists of $\phi(\pml{p})$ gaps, which becomes unwieldy around $p=37$:

\begin{center}
\begin{tabular}{rr} \hline
\multicolumn{1}{c}{$p$} & \multicolumn{1}{c}{length of $\pgap(\pml{p})$} \\ \hline
$31$ &  $3.06561 \; E10$ \\
$37$ &  $1.10362 \; E12$ \\
$41$ &  $4.41448 \; E13$ \\
$43$ &  $1.85408 \; E15$ \\ \hline
\end{tabular}
\end{center}

The eigenstructure of the dynamic system in Equation~(\ref{j1JEq}) enables us to extract
the asymptotic behavior of the population for constellations, if we can derive the initial conditions
for a cycle of gaps $\pgap(\pml{p})$ under much weaker conditions than those of 
Theorem~\ref{sThm}.

Extending the analogue of Polignac's conjecture for Eratosthenes sieve \cite{HRErat},
we are able to establish a strong Polignac result for arithmetic progressions:  that for
any $g=2n$, every feasible repetition $s=g,\ldots,g$ occurs in Eratosthenes sieve.
To the extent that these repetitions survive the sieve, they correspond to consecutive
primes in arithmetic progression with gap $g$.

\section{Modeling the populations of constellations}

\subsection{Summarizing the model for gaps}
In \cite{HRsmall} we developed a dynamic system that exactly models the populations of gaps
through the stages of Eratosthenes sieve.  We restate the model briefly here.

Let $g$ be a gap (a constellation of length $1$) with driving terms up to length $J$.
Let $p_0$ be a prime such that $g < 2p_1$.
Then
\begin{eqnarray*}
n_g (\pml{p_k}) & = & \left. M \right|_{p_k} \cdot n_g(\pml{p_{k-1}}) \\
 & = & \left[ \begin{array}{ccccc}
 p_k-2 & 1 & 0 & \cdots & 0 \\
 0 & p_k-3 & 2 & \cdots & 0 \\
 \vdots & & \ddots & \ddots & \vdots \\
 0 & \cdots & & & J-1 \\
 0 & \cdots & & & p_k-J-1
 \end{array} \right] \cdot n_g(\pml{p_{k-1}})  \\
 & & \\
 & = & R  \cdot \Lambda(p_k) \cdot L \cdot n_g(\pml{p_{k-1}})
\end{eqnarray*}
with 
\[ \Lambda(p_k) =  {\rm diag} \left[ p_k-2, \; p_k-3, \ldots, p_k-J-1 \right].
\]

\noindent Since the eigenvectors do not depend on the prime $p$, the dynamic system
can be expressed in terms of the initial conditions.
\begin{eqnarray}\label{ngEq}
n_g (\pml{p_k}) & = & \left. M \right|_{p_k} \cdots \left. M \right|_{p_1} \cdot n_g(\pml{p_0}) \\
 & = & R \cdot \Lambda^k  \cdot L \cdot n_g(\pml{p_0}) \nonumber
\end{eqnarray}
with
\[ \Lambda^k = {\rm diag} \left[ \prod_1^k (p_i-2), \;  
\prod_1^k (p_i-3), \ldots, \prod_1^k (p_i-J-1) \right].
\]

\noindent Dividing by $p-2$ at each iteration, we obtain the normalized population model
for gaps
\begin{eqnarray}\label{wgEq}
w_g (\pml{p_k}) & = & \left. M \right|_{p_k} \cdots \left. M \right|_{p_1} \cdot w_g(\pml{p_0}) \\
 & = & R \cdot \Lambda^k \cdot L \cdot w_g(\pml{p_0}) \nonumber 
\end{eqnarray}
and here
\[ \Lambda^k = {\rm diag} \left[ 1, \prod_1^k \frac{p_i-3}{p_i-2}, \; 
 \prod_1^k \frac{p_i-4}{p_i-2}, \ldots, \prod_1^k \frac{p_i-J-1}{p_i-2} \right].
 \]

The matrix $L$ is an upper triangular Pascal matrix, whose entries are the
binomial coefficients, and $R$ is an upper triangular matrix whose entries have
alternating signs.
\[
\begin{array}{lcl}
R_{ij} = \left\{ \begin{array}{crl}
(-1)^{i+j} \binom{j-1}{i-1} & {\rm if} & i \le j \\
 & & \\
 0 & {\rm if} & i > j \end{array} \right.
 & {\rm and} &
L_{ij} = \left\{ \begin{array}{crl}
 \binom{j-1}{i-1} & {\rm if} & i \le j \\
 & & \\
 0 & {\rm if} & i > j \end{array} \right.
\end{array}
\]

This dynamic system encompasses constellations as driving terms for gaps.  To adapt
the system to focus directly on the population of a constellation and its driving terms, we only
have to shift the eigenvalues.

\subsection{Extending the model to constellations}

The model for gaps relies on keeping track of the internal closures and external closures
for the driving terms for gaps.  For constellations, we need a richer concept than external closures,
and so we define the {\em boundary closures} for a driving term for a constellation.

Let $s$ be a constellation of length $j_1$, $s=g_1,g_2,\ldots,g_{j_1}$.
A driving term $\tilde{s}$ for $s$ will have the form $\tilde{s} = \tilde{s_1} \; \tilde{s_2} \ldots\tilde{s_{j_1}}$
in which $|\tilde{s_i}|=g_i$ for each $i$.

We call the
closures within an $\tilde{s_i}$ the {\em interior closures} for $\tilde{s}$, and the
exterior closures for each $\tilde{s_i}$ are the {\em boundary closures} for $\tilde{s}$.
Interior closures preserve the copy of $\tilde{s}$ as a driving term for $s$ but of
shorter length.  The $j_1+1$ boundary closures remove the copy of $\tilde{s}$
from being a driving term for $s$.

We are now ready to prove Theorem~\ref{sThm}.

\begin{proof} {\em of Theorem~\ref{sThm}:}
Let $s$ be a constellation of length $j_1$ and of sum $|s|$.
Let $p_0$ be a prime such that $|s| < 2p_1$.

Let $n_{s,j}(\pml{p})$ be the population of driving terms of length $j$ for $s$ in $\pgap(\pml{p})$,
with $j_1 \le j \le J$ and $p \ge p_0$.

Consider the recursion from $\pgap(\pml{p_{k-1}})$ to $\pgap(\pml{p_k})$, with $k \ge 1$.
Let $\tilde{s}$ be a particular occurrence of a driving term of length $j$ in $\pgap(\pml{p_{k-1}})$.
Since $|s| < 2p_k$, for the $p_k$ copies of $\tilde{s}$ initially created during step R2
of the recursion, 
the $j+1$ closures of step R3 all occur in different copies.  

Of the $p_k$ initial copies of $\tilde{s}$, the boundary closures eliminate $j_1+1$ copies
as driving terms for $s$.
The $j-j_1$ interior closures produce $j-j_1$ driving terms of length $j-1$, and 
$p_k-j-1$ copies of $\tilde{s}$ survive intact as driving terms of length $j$.

We can express this in the dynamic system:
\begin{eqnarray*}
n_s (\pml{p_k}) & = & M_{j_1:J}(p_k) \cdot n_s(\pml{p_{k-1}}) \\
 & = & \left[ \begin{array}{ccccc}
 p_k-j_1-1 & 1 & 0 & \cdots & 0 \\
 0 & p_k-j_1-2 & 2 & \cdots & 0 \\
 \vdots & & \ddots & \ddots & \vdots \\
 0 & \cdots & & & J-j_1 \\
 0 & \cdots & & & p_k-J-1
 \end{array} \right] \cdot n_g(\pml{p_{k-1}})  
 \end{eqnarray*}
Note that $M_{j_1:J}(p)$ is a $(J-j_1+1)\times(J-j_1+1)$ matrix.

The matrix $M_{j_1:J}(p)$ has the same eigenvectors as the matrix for gaps (of size
$J-j_1+1$).  The eigenstructure for $M_{j_1:J}(p)$ is given by
\[  M_{j_1:J}(p) = R \cdot \Lambda_{j_1:J}(p) \cdot L
\]
in which $R$ is still the upper triangular Pascal matrix with alternating signs, and $L$
is the upper triangular Pascal matrix.  These do not depend on the prime $p$. 
The eigenvalues are 
\[ \Lambda_{j_1:J}(p) =  {\rm diag} \left[ p-j_1-1, \; p-j_1-2, \ldots, p-J-1 \right].
\]

\noindent Since the eigenvectors do not depend on the prime $p$, the dynamic system
can easily be expressed in terms of the initial conditions.
\begin{eqnarray}\label{nsEq}
n_s (\pml{p_k}) & = & \left. M_{j_1:J} \right|_{p_k} \cdots \left. M \right|_{p_1} \cdot n_s(\pml{p_0}) \\
 & = & R \cdot \Lambda_{j_1:J}^k  \cdot L \cdot n_s(\pml{p_0}) \nonumber
\end{eqnarray}
with
\[ \Lambda_{j_1:J}^k = {\rm diag} \left[ \prod_1^k (p_i-j_1-1), \gap  
\prod_1^k (p_i-j_1-2), \ldots, \gap \prod_1^k (p_i-J-1) \right].
\]
\end{proof}

\subsection{Normalizing the populations}
In the large, the population of a constellation of length $j_1$ grows primarily by a factor of $(p-j_1-1)$.  So all constellations of length $j$ ultimately become more numerous than 
any constellation of length $j+1$, and comparing the asymptotic populations of constellations
of different lengths is thereby trivial.  

On the other hand, to determine the relative occurrence among constellations of a given length $j_1$, we divide by the factor $(p-j_1-1)$.  
We define the {\em normalized population} of a constellation $s$ of length $j_1$ and driving terms
up to length $J$ as
\begin{equation*}
w_s (\pml{p}) = \left( \prod_{q > j_1+1}^p \frac{1}{q-j_1-1} \right) \cdot n_s(\pml{p}) 
 \; = \; \frac{1}{\phi_{j_1+1}(\pml{p})} \cdot n_s(\pml{p}).
\end{equation*}
We introduce the functions $\phi_i(\pml{p}) = \prod_{q > i}^{p}(q-i).$

{\bf Definition.}  Let $Q=q_1 q_2 \cdots q_m$ be a product of distinct primes, with 
$q_1 < q_2 < \ldots < q_m$.  We define 
\[ \phi_{i}(Q) = \prod_{q_j > i} (q_j - i).\]
Note that $\phi_1 = \phi$, the Euler totient function, over the defined domain - products of distinct primes.  For the primorials we have
\[ \phi_{i}(\pml{p}) = \prod_{q>i}^{p} (q-i). \]

For the normalized populations, the dynamic system becomes
\begin{equation}\label{wsEq}
w_s (\pml{p_k}) =  R \cdot \Lambda_{j_1:J}^k  \cdot L \cdot w_s(\pml{p_0})
\end{equation}
with
\[ \Lambda_{j_1:J}^k = {\rm diag} \left[ 1, \gap \prod_1^k \frac{p_i-j_1-2}{p_i-j_1-1}, \gap  
\prod_1^k \frac{p_i-j_1-3}{p_i-j_1-1}, \ldots, \gap \prod_1^k \frac{p_i-J-1}{p_i-j_1-1} \right].
\]
For these normalized populations, the dominant eigenvalue is now $1$.

For gaps, this normalization corresponds nicely to taking the ratio of the population of the gap under consideration to the population of the gap $g=2$.  For $j_1=2$, we can again interpret this normalization as a ratio, in this case to the populations of constellations $s=24$ or $s=42$. These constellations have no driving terms longer than $j_1$, and $n_{24}(\pml{p}) = n_{42}(\pml{p})$, so there is no ambiguity.  This correlation between the normalization and ratios to constellations consisting of $2$'s and $4$'s begins to break down at $j_1=3$ and collapses completely at $j_1=6$.  For $j_1=3$, we have two constellations to choose from, $s=242$ or $s=424$, and there is ambiguity because $n_{424}(\pml{p}) = 2\cdot n_{242}(\pml{p})$.  For $j_1=4$ we have $s=2424$ and $s=4242$, and for $j_1=5$ we have $s=42424$.  However, by looking at $\pgap(\pml{5})$ we see that there are no constellations consisting only of $2$'s and $4$'s for length $j_1 \ge 6$.  

The normalization does not necessarily provide ratios of the populations of constellations of length $j_1$ to the population of a known constellation.  Instead, it provides relative populations to a (perhaps hypothetical) constellation without driving terms other than the constellation itself
and with population $\phi_{j_1+1}(\pml{p})$.

For example, for $j_1=3$ we have $p_0=5$ and $p_0-j_1-1=1$.  The symmetric
constellation $s=242$ has $w_{242}^\infty = 1$; it has no additional driving terms, it is symmetric,
but it occurs at the middle of the cycle $\pgap(\pml{5})$.  In contrast, the symmetric
constellation $s=424$ has $w_{424}^\infty = 2$, even though it too has no additional driving terms.

\subsection{Relative occurrence in the large.} 
Equation~(\ref{wsEq}) applies under the same conditions as Theorem~\ref{sThm} -- that $s$
is a constellation of length $j_1$ and driving terms up to length $J$, and $p_0$ is a prime
such that $|s| < 2p_1$.  

If we can get a count $n_s(\pml{p_0})$ for $s$ and all of its driving terms in $\pgap(\pml{p_0})$
with $|s| < 2p_1$, then we can apply the eigenstructure to obtain the
asymptotic number of occurrences $w_s^\infty$ of $s$ relative to other constellations
of length $j_1$ as $p_k \fto \infty$:
\begin{eqnarray*}
w_s^\infty & = & \lim_{p_k \fto \infty} w_{s,j_1}(\pml{p_k}) \\
 & = & \lim_{p_k \fto \infty} \frac{n_{s,j_1}(\pml{p_k})}{\phi_{j_1+1}(\pml{p_k})} \\
 & = & L_1 \cdot w_s(\pml{p_0})
\end{eqnarray*}

For this asymptotic result, we can soften the requirement of needing to work with 
$p_0$ such that $|s| < 2p_1$.  For the full dynamic system to apply, this condition
guarantees that for each occurrence $\tilde{s}$ of a driving term for 
$s$ in $\pgap(\pml{p_{k-1}})$, the closures in forming $\pgap(\pml{p_k})$ occur
in distinct copies of $\tilde{s}$.

Since $L_1 = [1, 1, \ldots, 1]$, to calculate $w_s^\infty$ we only need to know
that the $j_1+1$ closures that remove a copy of $\tilde{s}$ from being a driving term
occur in distinct copies.  Other closures may occur inside those copies as well, but we
need to know that under the recursion $p-j_1-1$ copies of $\tilde{s}$ survive as driving
terms for $s$.

\begin{theorem}\label{winfThm}
Let $s$ be a constellation of $j_1$ gaps
\[ s = g_1, \; g_2, \; , \ldots, \; g_{j_1}.
\]
Let $p_0$ be the highest prime that divides any of the intervals 
$g_i + \cdots + g_j$ with $1 \le i \le j \le j_1$.
Then for any $p \ge p_0$,
\[ w_s^\infty = L_1 \cdot w_s(\pml{p}).
\]
\end{theorem}

\begin{proof}
Let $\tilde{s}$ be a driving term for $s$.
From $\pgap(\pml{p_0})$ on, the closures that remove copies of $\tilde{s}$ from being
driving terms for $s$ all occur in distinct copies.  Thus for any subsequent $p_k$,
\[ \left| n_s(\pml{p_k}) \right| = (p_k - j_1-1) \left| n_s(\pml{p_{k-1}}) \right|.
\]
We cannot be certain about the lengths of all of the copies that survive as driving terms,
but we do know how the total population grows -- by exactly the factor $p_k-j_1-1$. and
so for all $p_k > p_0$, 
\[ w_s^\infty = L_1 \cdot w_s(\pml{p_k}).
\]
\end{proof}

\section{Specific constellations}
We now apply the population model to a few specific constellations.
Pursuing the twin prime conjecture, we study how constellations that correspond to twin primes
persist in Eratosthenes sieve.  One of these, $s=242$, has been studied extensively before
\cite{NicelyQuads, Rib, JLB}; this constellation corresponds to prime quadruplets, two
pairs of twin primes separated by a gap of $4$.  

The other family of constellations that
we study here are those corresponding to consecutive primes in arithmetic progression.
These constellations consist of the same gap repeated, such as $s=66$ or $s=666$.

\subsection{Constellations related to twin primes.}
Twin primes correspond to a gap $g=2$.  Here we are studying how gaps $g=2$
arise in Eratosthenes sieve, and we do not address how many of these might survive
the sieve to become gaps between primes (in this case between twin primes).

The gap $g=2$ arises in several interesting constellations.  The first, $s=242$,
corresponds to prime quadruplets.  Prime quadruplets are the densest occurrence of 
four primes in the large, two pairs of twin primes separated by a gap of $4$.
The constellation $s=242$ has no additional driving terms, $j_1=J=3$.  
We could use $p_0=3$, for which $n_{242}(\pml{3})=[1]$.  We have
\[ w_{242}^\infty = 1.
\]

\subsubsection{Constellation $s=2,10,2$}
The next constellation we consider is $s=2,10,2$, which corresponds
to two pairs of twin primes separated by
a gap of $10$.  Here $j_1=3$ and $J=4$, for the driving terms $2642$ and $2462$.
Using $p_0=7$, we count $n_{2,10,2}(\pml{7})=[2,6]$ for two occurrences of $s=2,10,2$ and
three occurrences each of the driving terms $2642$ and $2462$.
\begin{eqnarray*}
w_{2,10,2}(\pml{7}) & = & \frac{1}{\phi_4(\pml{7})} \left[ \begin{array}{c} 2 \\ 6 \end{array} \right] 
 \; = \; \left[ \begin{array}{c} 2 / 3 \\ 2 \end{array} \right] \\
 {\rm and} & & \\
 w_{2,10,2}^\infty & = & 8 /3
\end{eqnarray*}
This means that as $p_k \fto \infty$, the number of occurrences of $s=2,10,2$ 
in the sieve approaches 
$8/3$ times the number of occurrences of the constellation $242$.  Remember that these
weights $w^\infty$ are relative only to other constellations of the same length $j_1$.

\subsubsection{Constellation $s=2,10,2,10,2$}
The constellation $s=2,10,2,10,2$ corresponds to three pairs of twin primes with gaps
of $g=10$ separating them.  This constellation also contains two overlapping copies of
$2,10,2$, and two overlapping driving terms for the constellation $12,12$ to which we will
return when we look at arithmetic progressions.

For $s=2,10,2,10,2$, we have $j_1=5$ and $J=7$.  Since $|s|=26$, 
for initial conditions we have to use $\pgap(\pml{13})$.
\begin{eqnarray*}
n_{2,10,2,10,2}(\pml{13}) & = & \left[ \begin{array}{c} 52 \\ 44 \\ 48 \end{array} \right] \\
w_{2,10,2,10,2}(\pml{13}) & = & \frac{1}{\phi_6(\pml{13})}
 \left[ \begin{array}{c} 52 \\ 44 \\ 48 \end{array} \right] \; = \;
 \frac{1}{35}
 \left[ \begin{array}{c} 52 \\ 44 \\ 48 \end{array} \right]  \\
{\rm and} & & \\
w_{2,10,2,10,2}^\infty & = & 144/35.
\end{eqnarray*}
So among constellations of length $5$, the constellation $s=2,10,2,10,2$
occurs with a relative frequency of $144/35$.  For length $j_1=5$, we can use the
constellation $42424$ as a reference.  The population model shows that in the large,
the constellation $s=2,10,2,10,2$ occurs over four times as frequently as the
constellation $42424$.

\subsubsection{Constellation $s=2,10,2,10,2,4,2,10,2,10,2$}  
Along this line of inquiry into constellations that contain several $2$s, we observe that 
the following constellation occurs in $\pgap(\pml{13})$:
\[ s = 2,10,2,10,2,4,2,10,2,10,2.
\]
This is two copies of $2,10,2,10,2$ separated by a gap of $4$.  This corresponds
to six pairs of twin primes, or twelve primes total, occurring in an interval of $|s|=56$.

For this constellation $s$, we have $j_1=11$ and $J=13$.  It would theoretically be possible
for each of the $10$'s to be produced through closures, so there could be driving terms of 
length up to $15$, but when we inspect $\pgap(\pml{13})$, we find that there are two copies of
$s$, ten driving terms of length $12$, twelve driving terms of length $13$, and no driving terms
of length $14$ or $15$.

Since $|s|=56$, to use Theorem~\ref{sThm} we would need to use $p_0=23$ to employ the 
full dynamic system.  However, to obtain the asymptotic results of Theorem~\ref{winfThm}
we can use $p_0=13$.  We calculate $w_s^\infty = 24$.  This means that relative to a
constellation of length $11$ with one occurrence in $\pgap(\pml{13})$ and no 
additional driving terms, as the sieve 
continues the constellation $s$ will occur approximately $24$ times as often.

\begin{table}
\begin{tabular}{c|rrr|rrr} \hline
$s$ & $|s|$ & $j_1$ & $J$ & $p_0$ & $n(\pml{p_0})$ & $\omega^\infty$ \\ \hline
$242$ & $8$ & $3$ & $3$ & $5$ & $[1]$ & $1$ \\
$424$ & $10$ & $3$ & $3$ & $5$ & $[2]$ & $2$ \\
$2,10,2$ & $14$ & $3$ & $4$ & $7$ & $[2,6]$ & $8/3$ \\
$42424$ & $16$ & $5$ & $5$ & $7$ & $[1]$ & $1$ \\
$2,10,2,10,2$ & $26$ & $5$ & $7$ & $13$ & $[52, 44, 48]$ & $144/35$ \\ 
$2,10,2,10,2,4,2,10,2,10,2$ & $56$ & $11$ & $13$ & $13$ & $[2, 10, 12]$ & $24$ \\
$66$ & $12$ & $2$ & $4$ & $5$ & $[0, 2, 2]$ & $2$ \\
$12,12$ & $24$ & $2$ & $6$ & $11$ & $[0, 2, 20, 48, 58]$ & $2$ \\
$666$ & $18$ & $3$ & $5$ & $7$ & $[0, 4, 2 ]$ & $2$ \\ \hline
\end{tabular}
\caption{Table of initial conditions and parameters for a few representative constellations.
The population of a constellation of length $j_1$ grows primarily by a factor of
$p-j_1-1$.}
\end{table}

\subsection{Consecutive primes in arithmetic progression}
A sequence of $j_1+1$ consecutive primes in arithmetic progression corresponds 
to a constellation of $j_1$ identical gaps $g$.
By considering residues, we easily see that for a sequence of $j_1+1$ primes in arithmetic 
progression, $g$ must be divisible by every prime $p \le j_1+1$.  
So for three consecutive primes in arithmetic progression, the minimal constellation is $s=66$.  
For four consecutive primes in arithmetic progression, the minimal constellation is $s=666$, and
then for an arithmetic progression of five consecutive primes the minimal constellation is 
$s=30,30,30,30$.  

\subsubsection{Constellation $s=66$.}
Let's now calculate the population of the constellation $s=66$.  This constellation 
corresponds to three consecutive primes in arithmetic progression: $p, \; p+6, \; p+12$.
Since $|s|=12$, we can still use $p_0 = 5$.  In $\pgap(\pml{5})=64242462$, we observe the
following initial conditions for $s=66$:
$$ n_{66}(\pml{5}) = \left[ \begin{array}{c}
 0 \\ 2 \\ 2 \end{array}\right].$$
For the first entry, we don't
yet have any occurrences of $s=66$.  We do have two driving terms of length three: $642$
and $246$; and two driving terms of length four: $4242$ and $2424$.

\begin{eqnarray*}
w_{66}(\pml{5}) & = & \frac{1}{\phi_3(\pml{5})}
 \left[ \begin{array}{c} 0 \\ 2 \\ 2 \end{array} \right] \\
{\rm and} & & \\
w_{66}^\infty & = & 2.
\end{eqnarray*}

\subsubsection{Constellation $s=12,12$.}
The constellation $s=12,12$ also
corresponds to three consecutive primes in arithmetic progression: $p, \; p+12, \; p+24$.
Since $|s|=24$, to apply the full dynamic system of Theorem~\ref{sThm} we use $p_0 = 11$.  
In $\pgap(\pml{11})$, we calculate the following, for driving terms from length $j_1=2$ to $J=6$:
\[ \begin{array}{lcl} n_{12,12}(\pml{11}) = \left[ \begin{array}{c}
0 \\ 2 \\ 20 \\ 48 \\ 58 \end{array}\right] &{\rm and} &
w_{12,12}(\pml{11}) =  \frac{1}{\phi_3(\pml{11})}
 \left[ \begin{array}{c} 0 \\ 2 \\ 20 \\ 48 \\ 58 \end{array} \right] 
 \end{array} 
 \]
 So for $s=12,12$ we have
\[
w_{12,12}^\infty = \frac{128}{8\cdot 4 \cdot 2} = 2.
\]
It is interesting that the asymptotic relative population of $s=12,12$ is the same as for the
constellation $s=66$.

\subsubsection{Constellation $s=666$.}
The constellation $s=666$ is the smallest constellation
corresponding to four consecutive primes in arithmetic progression.
Since $|s|=18$, we can use $p_0 = 7$.  In $\pgap(\pml{7})$, we have the
following initial conditions for $s=666$, for driving terms from length $j_1=3$ to $J=5$:
\[ \begin{array}{lcl} n_{666}(\pml{7}) = \left[ \begin{array}{c}
0 \\ 4 \\ 2 \end{array}\right] &{\rm and} &
w_{666}(\pml{7}) =  \frac{1}{\phi_4(\pml{7})}
 \left[ \begin{array}{c} 0 \\ 4 \\ 2  \end{array} \right] 
 \end{array} 
 \]
For $s=666$ we have
\[
w_{666}^\infty = \frac{6}{3} = 2.
\]
It is interesting that we again have the asymptotic relative population of $w_s^\infty=2$, 
although here it is relative to constellations of length $j_1=3$.

\section{Polignac result for arithmetic progressions}
Here we establish a strong Polignac result on arithmetic progressions.
In \cite{HRErat} we established the following analogue of the Polignac conjecture:

\begin{theorem}
Let $g=2n$ and let $Q= q_1 \cdots q_m$ be the product of the distinct prime factors
of $g$.  Then $g$ arises as a gap in Eratosthenes sieve and asymptotic ratio
\[ w_g^\infty = \prod_{2 < q_i} \frac{q_i - 1}{q_i - 2}.
\]
\end{theorem}

For a gap $g=2n$, we can now show that not only does the gap occur in Eratosthenes sieve,
but that every feasible repetition of $g$ as a constellation $s=g,\ldots , g$ occurs in
the sieve.  We make this precise below. 

Above, we calculated the occurrences for a few small examples.  We cannot
perform the brute force calculation for five primes in arithmetic progression.  The minimal
constellation we would need to consider is
$s=30,30,30,30$.  To apply Theorem~\ref{sThm} we would have to use $p_0=57$.  However,
we could obtain asymptotic results via Theorem~\ref{winfThm} with $p_0=5$.
The steps we would take in order to apply Theorem~\ref{winfThm} to this constellation
can be generalized to prove Theorem~\ref{PolsThm} below.

{\bf Definition.}  Let $g$ be a gap and let $p_k$ be the
largest prime such that $\pml{p_k} | g$.  Let $s= g,\ldots,g$ be a repetition of $g$ of length $j_1$.
Then the constellation $s$ is {\em feasible} iff $j_1 < p_{k+1}-1$.

If it survives subsequent stages of the sieve, a repetition $s$ of length $j_1$ corresponds to $j_1+1$ consecutive primes in arithmetic progression.

\begin{theorem}\label{PolsThm}
Let $g$ be an even number, 
and let $Q= q_1 \cdots q_m$ be the product of the distinct prime factors of $g$, with
$q_1 < \ldots < q_m$.  Let $s$ be a feasible repetition of $g$ of length $j_1$.
Then $s$ occurs in Eratosthenes sieve with asymptotic weight
\[ w_{g,\ldots,g}^\infty = \frac{\phi_1(Q)}{\phi_{j_1+1}(Q)}.
\]
\end{theorem}

\begin{proof}
We start in the cycle of gaps $\pgap(Q)$, in which $Q$ may not be a primorial. 
$\pgap(Q)$ consists of $\phi(Q)$ gaps that sum to $Q$.  So we can start with
any of the $\phi(Q)$ gaps and continue through the cycle $g/Q$ times, and
this concatenation is a driving term for $g$.  

We repeat this concatenation of cycles $\pgap(Q)$, to identify the driving terms
for $s=g,\ldots,g$, a feasible repetition of the gap $g$ of length $j_1$. 
Starting with any gap in $\pgap(Q)$, we continue through
$\pgap(Q)$ for $j_1 \cdot g / Q$ complete cycles, and this is
a driving term $\tilde{s}$ for $s$.
Since we can start from any gap in $\pgap(Q)$, we have $\phi(Q)$ driving terms for $s$.

Now that we have seeded the construction, we move from
$\pgap(Q)$ back into the cycles of gaps for the primorials, $\pgap(\pml{q_m})$.   

{\em Case 1. $Q$ is itself a primorial.}  In this case $Q=\pml{q_m}$, and the total number of driving 
terms for $s$ equals $L_1 \cdot n_s(\pml{q_m}) = \phi(Q)$.  From this, we calculate
the asymptotic ratio
\[ w_s^\infty = \frac{\phi(\pml{q_m})}{\phi_{j_1+1}(\pml{q_m})} = \frac{\phi_1(Q)}{\phi_{j_1+1}(Q)}.
\]

{\em Case 2. $Q$ is not a primorial.}  Let $\rho_1 < \rho_2 < \cdots < \rho_k$ be the
primes less than $q_m$ that are not factors of $Q$.  Let $Q_0=Q$.  For $i=1,\ldots,k$,
let $Q_i = Q_{i-1} \cdot \rho_i.$  

Let $\tilde{s}$ be a driving term for $s$ in $\pgap(Q_{i-1})$.  Under the recursion to create
$\pgap(Q_i)$, we create $\rho_i$ copies of $\tilde{s}$ in step R2.  Then in step R3 we close
gaps as indicated by the element wise product $\rho_i * \pgap(Q_{i-1})$.  

The driving term $\tilde{s}$ is composed of $j_1$ driving terms for $g$.
\[ \tilde{s} = \tilde{s_1} \; \tilde{s_2} \; \ldots \; \tilde{s_{j_1}}
\]
in which each $|\tilde{s_i}| = g$.  

The $j_1+1$ boundary closures will eliminate copies of $\tilde{s}$
from being a driving term for $s$.
All of the intervals in Theorem~\ref{winfThm} are multiples of
$g$, and these multiples themselves have prime factors entirely divisible by the factors of $g$.
That is, if $\pml{p}$ is the largest primorial that divides $g$, then for all $2 \le j \le j_1$, 
by the feasibility of $s$, the prime factors of $j$ are factors of $\pml{p}$.

Since $\rho_i \not| g$, all of the boundary closures for $\tilde{s}$ occur in different copies
of $\tilde{s}$.  Thus, of the $\rho_i$ initial copies of $\tilde{s}$, $j_1+1$ are removed as
driving terms for $s$, and the other $\rho_i-j_1-1$ copies remain as driving terms of
some length but no longer than the length of $\tilde{s}$.  

We don't know the lengths of the driving terms for $s$ in $\pgap(Q_i)$, but we do know the
total population:
\begin{eqnarray*}
L_1 \cdot n_s(Q_i) & = & (\rho_i - j_1 - 1) \cdot L_1 \cdot n_s(Q_{i-1}) \\
 & = & (\rho_i-j_1-1)(\rho_{i-1}-j_1-1)\cdots(\rho_1-j_1-1) \cdot L_1 \cdot n_s(Q_0) \\
 & = & (\rho_i-j_1-1)(\rho_{i-1}-j_1-1)\cdots(\rho_1-j_1-1) \cdot \phi(Q).
\end{eqnarray*}
Continuing this construction, we have $Q_k=\pml{q_m}$, from which
\[ L_1 \cdot n_s(Q_k) = (\rho_k-j_1-1)(\rho_{k-1}-j_1-1)\cdots(\rho_1-j_1-1) \cdot \phi(Q),
\]
and the asymptotic ratio is
\begin{equation*}
 w_s^\infty = \frac{1}{\phi_{j_1+1}(\pml{q_m})} \cdot L_1 \cdot n_s(\pml{q_m}) \; = \;
 \frac{\phi(Q)}{\phi_{j_1+1}(Q)}.
\end{equation*}
\end{proof}

Thus every feasible repetition of a gap $g$ occurs in Eratosthenes sieve.  Compared
to other constellations of length $j_1$, the repetition of $g$ of length $j_1$ has asymptotic
weight $\phi_1(Q) / \phi_{j_1+1}(Q)$.  To the extent that these constellations survive
the sieve, they correspond to $j_1+1$ consecutive primes in arithmetic progression.

\section{Conclusion}
We have adapted the dynamic model for populations of gaps, as developed
in \cite{HRErat}, to apply to populations of constellations.  These models for 
constellations all have the same left and right eigenvectors.  The matrix of right eigenvectors
$R$ is the upper triangular Pascal matrix with alternating signs, and the 
matrix of left eigenvectors $L$ is the upper triangular Pascal matrix.
The eigenvalues depend on the prime and on the length of the constellation, but
these otherwise have the same form.

We have calculated the populations for certain specific constellations containing
the gap $2$.  These correspond to pairs of twin primes in various constellations.

We also calculated the populations for a few constellations that are the repetition
of a gap $g$.  These repetitions correspond to consecutive primes in arithmetic 
progression.  We are able to establish a strong Polignac result for arithmetic progressions:
for any $g=2n$, every feasible repetition of $g$ occurs as a constellation in 
Eratosthenes sieve.  The asymptotic ratio for the occurrence of the repetition of
$g$ of length $j_1$ depends only on the prime factors of $g$ and the length $j_1$.


\bibliographystyle{amsplain}

\providecommand{\bysame}{\leavevmode\hbox to3em{\hrulefill}\thinspace}
\providecommand{\MR}{\relax\ifhmode\unskip\space\fi MR }
\providecommand{\MRhref}[2]{%
  \href{http://www.ams.org/mathscinet-getitem?mr=#1}{#2}
}
\providecommand{\href}[2]{#2}

\end{document}